\documentclass[11pt, reqno]{amsart}
\usepackage{amsmath, amssymb, amsthm}
\usepackage{geometry}
\usepackage{graphicx}
\usepackage{microtype}
\usepackage[hidelinks]{hyperref} 

\newtheorem{theorem}{Theorem}[section]
\newtheorem{lemma}[theorem]{Lemma}
\newtheorem{proposition}[theorem]{Proposition}
\newtheorem{corollary}[theorem]{Corollary}

\theoremstyle{definition}
\newtheorem{definition}[theorem]{Definition}
\newtheorem{example}[theorem]{Example}

\theoremstyle{remark}
\newtheorem{remark}[theorem]{Remark}

\newcommand{\R}{\mathbb{R}}

\title[Almost RB solitons on contact metric 3-manifolds]
{Almost Ricci--Bourguignon Solitons on Contact Metric Three-Manifolds}
\author{Mohammad Aqib}
\address{Harish-Chandra Research Institute A CI of Homi Bhabha National
Institute Chhatnag Road, Jhunsi, Prayagraj-211019, India.}
\address{Homi Bhabha National Institute, Training School Complex, Anushakti Nagar, Mumbai 400094, India}
\email{mohammadaqib@hri.res.in}



\begin{document}

\begin{abstract}
We investigate almost Ricci--Bourguignon solitons on three-dimensional contact
metric manifolds. Under natural curvature assumptions, we show that the
additional freedom introduced by allowing the soliton function to vary is
rigidly constrained by the contact geometry. Using a local orthonormal \(\varphi\)-basis on the
non-Sasakian region, we derive the full component form of the almost
Ricci--Bourguignon soliton equation. As applications, we consider the cases
where the potential vector field is pointwise collinear with, or orthogonal to,
the Reeb vector field. For contact metric three-manifolds satisfying
\(Q\xi=\sigma\xi\), we prove that a collinear potential field must vanish on
the non-Sasakian region whenever \(\xi(\sigma)=0\). In the orthogonal case, when \(\sigma\) is constant and the manifold is
non-Sasakian, the almost soliton function is forced to be constant; hence the
soliton reduces to a Ricci--Bourguignon soliton. In fact, the metric is
Einstein, and if the orthogonal potential field is not identically zero, then the
metric is flat.
\end{abstract}

\maketitle

\section{Introduction}
Ricci solitons arise as self-similar solutions of the Ricci flow and play an
important role in the study of geometric evolution equations. The
Ricci--Bourguignon flow, introduced by Bourguignon, is defined by
\[
\frac{\partial g}{\partial t}
=
-2(\operatorname{Ric}-\rho r\,g),
\]
where \(r\) is the scalar curvature and \(\rho\in\mathbb R\) is a constant.
For \(\rho=0\), this reduces to the Ricci flow. The analytic theory of the
Ricci--Bourguignon flow, including short-time existence under suitable
conditions on \(\rho\), was developed in \cite{catino2017rbflow}.

The notion of Ricci almost soliton was introduced by Pigola, Rigoli, Rimoldi
and Setti \cite{pigola2011ricci}, by allowing the soliton constant to be a
smooth function. Several extensions and applications of almost Ricci solitons
have since been studied; see, for example, \cite{guler,blaga,MR2869087}.
In the Ricci--Bourguignon setting, almost solitons were considered by
Dwivedi \cite{dwivedi2021rb}.

\begin{definition}
    An {\it almost Ricci--Bourguignon soliton} (almost RB soliton for short) is a Riemannian manifold $(M^n, g)$ endowed with a vector field $V$ on $M$ and a soliton function $\lambda: M\to\mathbb{R}$ that satisfies
    \begin{equation}\label{rbsoliton}
     \operatorname{Ric}+\frac{1}{2} \mathcal{L}_{V} g=(\lambda +\rho\, r) g,
 \end{equation}
where $\mathcal{L}_V g$ denotes the Lie derivative of the metric $g$ with respect to the vector field $V$,
$r$ denotes the scalar curvature of the almost RB soliton
and $\rho \in \R$ is a constant.  
\end{definition}

Again, if the vector field $V$ is a gradient of some smooth function $f$ on $M$, then it is called a gradient almost RB soliton. In this case, the foregoing equation reduces to
\begin{equation}\label{grbsoliton}
\nabla^{2} f + \operatorname{Ric}=(\lambda+\rho  r) g 
\end{equation}

where $\nabla^{2} f$ stands for the Hessian of $f$. If the potential function \(f\) is constant, then the gradient soliton is
trivial. Perelman \cite{perelman2002} proved that the Ricci solitons on a compact manifold are always gradient Ricci solitons and, hence, the potential vector field can be expressed as the sum of the gradient of a function and a Killing vector field. Such a result is also true for the almost RB solitons if the scalar curvature is constant (see \cite{Aqib}). For further details on almost RB solitons, refer to \cite{dwivedi2021rb} and \cite{ghosh2022rb}.

Ricci--Bourguignon solitons on three-dimensional contact metric manifolds
were recently studied by Khatri and Singh \cite{khatri2024rbcontact}. Their
work concerns the case where the soliton parameter is constant. In contrast,
we consider the almost Ricci--Bourguignon equation, where the soliton function
\(\lambda\) is allowed to vary. Our results show that, under the curvature
condition \(Q\xi=\sigma\xi\), this extra freedom is strongly restricted by the
contact metric structure. In particular, in the Reeb-orthogonal non-Sasakian
case with \(\sigma\) constant, the soliton function is forced to be constant,
so the almost soliton reduces to a Ricci--Bourguignon soliton.

The main results may be summarised as follows. If \(Q\xi=\sigma\xi\) and
\(\xi(\sigma)=0\), then any almost Ricci--Bourguignon soliton with potential
field \(V=f\xi\) satisfies \(f=0\) on the non-Sasakian region. Thus, if \(V\) is
nowhere vanishing, the manifold must be Sasakian. In the Reeb-orthogonal
non-Sasakian case, when \(\sigma\) is constant, the \((\xi,\xi)\)-component
forces
\[
\lambda+\rho r=\sigma.
\]
Consequently, the soliton function is constant and almost
Ricci--Bourguignon soliton reduces to a Ricci--Bourguignon soliton. The
reduced component system then implies that the metric is Einstein; if the
orthogonal potential field is not identically zero, the metric is flat.

The present work complements the Ricci--Bourguignon soliton classifications of
Khatri and Singh \cite{khatri2024rbcontact} by considering the almost soliton
setting.

\section{Preliminaries}
A \( (2n+1) \)-dimensional Riemannian manifold \( (M, g) \) is called a contact metric manifold if it admits:
\begin{itemize}
    \item A global 1-form \( \eta \) such that \( \eta \wedge (d\eta)^n \neq 0 \) at every point on \( M \) (the contact 1-form),
    \item  A unit vector field \( \xi \), called the Reeb vector field, satisfying \( \eta(\xi) = 1 \) and \( d\eta(\xi, \cdot) = 0 \),
    \item A (1,1)-tensor field \( \varphi \) such that
\end{itemize}
\begin{align}\label{2.1c}
    d \eta(X, Y)&=g(X, \varphi Y), \quad \eta(X)=g(X, \xi),\quad \varphi^{2} X=-X+\eta(X) \xi
\end{align}

From the above, it follows that:
\begin{align}\label{2.1f}
    \varphi \xi&=0,  \quad \eta \circ \varphi=0,\quad    g(\varphi X, \varphi Y)=g(X, Y)-\eta(X) \eta(Y).
\end{align}

A Riemannian manifold \( M^{2n+1} \) equipped with \( (\varphi, \xi, \eta, g) \) is called a contact metric manifold. Next, we define two self-adjoint operators $h$ and $l$ by $h=\frac{1}{2} \mathcal{L}_{\xi} \varphi$ and $l=R(., \xi) \xi$. These tensors $h$ and $h \varphi$ are trace-free and satisfy $h \varphi=-\varphi h$. For a contact metric manifold, the following formulas hold \cite{MR1874240}
\begin{align}\label{2.1b}
\nabla_{X} \xi  =-\varphi X-\varphi h X  \quad
g(Q \xi, \xi)  =\operatorname{Trl}=2 n-|h|^{2} 
\end{align}

If \( \xi \) is a Killing vector field (\( h = 0 \) or \( \text{Tr}(l) = 2n \)), the contact metric manifold is called a \(K\)-contact manifold. On a $K$-contact manifold, the following formulas are known \cite{MR1874240}
\begin{align}\label{2.3b}
\nabla_{X} \xi  =-\varphi X,  \quad
Q \xi  =2 n \xi, \quad
R(X, \xi) \xi  =X-\eta(X) \xi, 
\end{align}

 A contact metric structure is said to be Sasakian if the metric cone $C(M)\left(d r^{2}+r^{2} g, d\left(r^{2} \eta\right)\right)$ is Kähler (see \cite{MR1823927}). The formulas from equations \eqref{2.3b} are also valid for a Sasakian manifold. Moreover, a contact metric manifold is Sasakian if
\begin{equation}
    \left(\nabla_{X} \varphi\right) Y=g(X, Y) \xi-\eta(Y) X,
\end{equation}
or equivalently
\begin{equation}\label{2.6b}
R(X, Y) \xi=\eta(Y) X-\eta(X) Y 
\end{equation}

A Sasakian manifold is $K$-contact, and the converse is not true except in dimension 3.

\subsection{Three-dimensional contact metric manifolds}

Let \((M^3,\varphi,\xi,\eta,g)\) be a three-dimensional contact metric
manifold. On the open set
\[
U=\{p\in M:h_p\neq0\},
\]
there exists a local orthonormal \(\varphi\)-basis
\[
\{e,\varphi e,\xi\}
\]
such that
\[
he=\ell e,\qquad h\varphi e=-\ell\varphi e,
\]
where \(\ell>0\) is a smooth function on \(U\). In this basis the Levi-Civita
connection is given by
\begin{align}
\nabla_\xi e&=a\varphi e, &
\nabla_\xi\varphi e&=-ae, &
\nabla_\xi\xi&=0, \label{pre-conn1}\\
\nabla_e\xi&=-(1+\ell)\varphi e, &
\nabla_e e&=b\varphi e, &
\nabla_e\varphi e&=-be+(1+\ell)\xi, \label{pre-conn2}\\
\nabla_{\varphi e}\xi&=(1-\ell)e, &
\nabla_{\varphi e}\varphi e&=ce, &
\nabla_{\varphi e}e&=-c\varphi e+(\ell-1)\xi. \label{pre-conn3}
\end{align}
Moreover,
\[
b=\frac{(\varphi e)(\ell)+A}{2\ell},\qquad
c=\frac{e(\ell)+B}{2\ell},
\]
where
\[
A=\operatorname{Ric}(e,\xi),\qquad
B=\operatorname{Ric}(\varphi e,\xi).
\]
The Ricci operator is given by
\begin{align}
Qe
&=
\left(\frac r2-1+\ell^2-2a\ell\right)e
+Z\varphi e
+A\xi, \label{pre-ric1}\\
Q\varphi e
&=
Ze+
\left(\frac r2-1+\ell^2+2a\ell\right)\varphi e
+B\xi, \label{pre-ric2}\\
Q\xi
&=
Ae+B\varphi e+2(1-\ell^2)\xi, \label{pre-ric3}
\end{align}
where
\[
Z=\xi(\ell).
\]
These formulas are standard in the study of three-dimensional contact metric
manifolds; see, for example, Koufogiorgos \cite{koufogiorgos1995} and
Chen \cite{chen2021cotton}.

\section{Main Results}

The next proposition provides a complete local decomposition of the almost Ricci--Bourguignon soliton equation in an adapted \(\varphi\)-frame. This formulation is the main technical tool used to prove the rigidity results for Reeb-aligned and Reeb-orthogonal potentials in the remainder of the paper.

\begin{proposition}[Local decomposition of the almost Ricci--Bourguignon soliton equation]
\label{prop:component-system}
Let \((M^3,\varphi,\xi,\eta,g)\) be a three-dimensional contact metric manifold.
On the open set
\[
U=\{p\in M:h_p\neq 0\},
\]
choose a local orthonormal \(\varphi\)-basis
\[
\{e,\varphi e,\xi\}
\]
such that
\[
he=\ell e,\qquad h\varphi e=-\ell\varphi e,
\]
where \(\ell>0\) is a smooth function on \(U\). Let
\[
V=f_1e+f_2\varphi e+f_3\xi
\]
be a vector field on \(U\), and put
\[
\Lambda=\lambda+\rho r.
\]
Then the almost Ricci--Bourguignon soliton equation
\[
\operatorname{Ric}+\frac12\mathcal L_Vg=\Lambda g
\]
is equivalent to the following system:
\begin{align}
\frac r2-1+\ell^2-2a\ell+e(f_1)-bf_2
&=\Lambda, \label{S1}\\
\frac r2-1+\ell^2+2a\ell+(\varphi e)(f_2)-cf_1
&=\Lambda, \label{S2}\\
2(1-\ell^2)+\xi(f_3)
&=\Lambda, \label{S3}\\
Z+\frac12\left\{e(f_2)+(\varphi e)(f_1)+bf_1+cf_2-2\ell f_3\right\}
&=0, \label{S4}\\
A+\frac12\left\{e(f_3)+\xi(f_1)+(1+\ell-a)f_2\right\}
&=0, \label{S5}\\
B+\frac12\left\{(\varphi e)(f_3)+\xi(f_2)+(a+\ell-1)f_1\right\}
&=0. \label{S6}
\end{align}
Here
\[
A=\operatorname{Ric}(e,\xi),\qquad
B=\operatorname{Ric}(\varphi e,\xi),\qquad
Z=\xi(\ell).
\]
\end{proposition}

\begin{proof}
Let
\[
V=f_1e+f_2\varphi e+f_3\xi.
\]
Using \eqref{pre-conn1}--\eqref{pre-conn3}, we compute
\begin{align}
\nabla_eV
={}&\bigl(e(f_1)-bf_2\bigr)e
+\bigl(bf_1+e(f_2)-(1+\ell)f_3\bigr)\varphi e  \nonumber\\
&+\bigl((1+\ell)f_2+e(f_3)\bigr)\xi, \label{nablaeV}\\
\nabla_{\varphi e}V
={}&\bigl((\varphi e)(f_1)+cf_2+(1-\ell)f_3\bigr)e
+\bigl(-cf_1+(\varphi e)(f_2)\bigr)\varphi e \nonumber\\
&+\bigl((\ell-1)f_1+(\varphi e)(f_3)\bigr)\xi, \label{nablaphieV}\\
\nabla_\xi V
={}&\bigl(\xi(f_1)-af_2\bigr)e
+\bigl(af_1+\xi(f_2)\bigr)\varphi e
+\xi(f_3)\xi. \label{nablaxiV}
\end{align}
Therefore,
\[
\frac12(\mathcal L_Vg)(X,Y)
=
\frac12\{g(\nabla_XV,Y)+g(\nabla_YV,X)\}.
\]
Computing the symmetric part of the covariant derivative of \(V\), we obtain,
\begin{align}
\frac12(\mathcal L_Vg)(e,e)
&=e(f_1)-bf_2, \label{lie1}\\
\frac12(\mathcal L_Vg)(\varphi e,\varphi e)
&=(\varphi e)(f_2)-cf_1, \label{lie2}\\
\frac12(\mathcal L_Vg)(\xi,\xi)
&=\xi(f_3), \label{lie3}\\
\frac12(\mathcal L_Vg)(e,\varphi e)
&=\frac12\left\{e(f_2)+(\varphi e)(f_1)+bf_1+cf_2-2\ell f_3\right\}, \label{lie4}\\
\frac12(\mathcal L_Vg)(e,\xi)
&=\frac12\left\{e(f_3)+\xi(f_1)+(1+\ell-a)f_2\right\}, \label{lie5}\\
\frac12(\mathcal L_Vg)(\varphi e,\xi)
&=\frac12\left\{(\varphi e)(f_3)+\xi(f_2)+(a+\ell-1)f_1\right\}. \label{lie6}
\end{align}

From \eqref{pre-ric1}--\eqref{pre-ric3}, the Ricci components are
\begin{align}
\operatorname{Ric}(e,e)
&=\frac r2-1+\ell^2-2a\ell, \label{ric1}\\
\operatorname{Ric}(\varphi e,\varphi e)
&=\frac r2-1+\ell^2+2a\ell, \label{ric2}\\
\operatorname{Ric}(\xi,\xi)
&=2(1-\ell^2), \label{ric3}\\
\operatorname{Ric}(e,\varphi e)
&=Z, \label{ric4}\\
\operatorname{Ric}(e,\xi)
&=A, \label{ric5}\\
\operatorname{Ric}(\varphi e,\xi)
&=B. \label{ric6}
\end{align}

Substituting \eqref{lie1}--\eqref{lie6} and \eqref{ric1}--\eqref{ric6}
into
\[
\operatorname{Ric}+\frac12\mathcal L_Vg=\Lambda g
\]
for the six independent pairs
\[
(e,e),\quad
(\varphi e,\varphi e),\quad
(\xi,\xi),\quad
(e,\varphi e),\quad
(e,\xi),\quad
(\varphi e,\xi),
\]
we obtain precisely \eqref{S1}--\eqref{S6}.
Conversely, if equations \eqref{S1}--\eqref{S6} hold, then all components of
the tensor
\[
\operatorname{Ric}+\frac12\mathcal L_Vg-\Lambda g
\]
vanish with respect to the orthonormal basis \(\{e,\varphi e,\xi\}\).
Hence, the tensor identity follows.
\end{proof}

The significance of Proposition \ref{prop:component-system} is that it converts the almost Ricci--Bourguignon soliton equation into a first-order system whose structure makes the interaction between the potential field and the Reeb direction completely explicit. The rigidity results proved below arise from imposing geometric conditions on this system.

\begin{corollary}[Trace equation]
\label{cor:trace-equation}
Under the assumptions of Proposition \ref{prop:component-system}, one has
\[
\operatorname{div}V=3\Lambda-r.
\]
Equivalently,
\[
e(f_1)+(\varphi e)(f_2)+\xi(f_3)-cf_1-bf_2
=
3\lambda+(3\rho-1)r.
\]
\end{corollary}

\begin{proof}
Taking the trace of
\[
\operatorname{Ric}+\frac12\mathcal L_Vg=\Lambda g
\]
gives
\[
r+\operatorname{div}V=3\Lambda.
\]
Hence
\[
\operatorname{div}V=3\Lambda-r.
\]
Using the expressions for \(\nabla_e V\), \(\nabla_{\varphi e} V\) and \(\nabla_\xi V\) obtained in the proof of Proposition \ref{prop:component-system}, we compute
\[
\operatorname{div}V
=
g(\nabla_eV,e)+g(\nabla_{\varphi e}V,\varphi e)+g(\nabla_\xi V,\xi),
\]
and therefore
\[
\operatorname{div}V
=
e(f_1)+(\varphi e)(f_2)+\xi(f_3)-cf_1-bf_2.
\]
This proves the claim.
\end{proof}

\begin{remark}[The case \(V=f\xi\)]
\label{rem:V-collinear-xi}
If \(V\) is pointwise collinear with the Reeb vector field, say
\[
V=f\xi,
\]
then
\[
f_1=f_2=0,\qquad f_3=f.
\]
The system \eqref{S1}--\eqref{S6} reduces to
\begin{align}
\frac r2-1+\ell^2-2a\ell&=\Lambda, \label{col1}\\
\frac r2-1+\ell^2+2a\ell&=\Lambda, \label{col2}\\
2(1-\ell^2)+\xi(f)&=\Lambda, \label{col3}\\
Z-\ell f&=0, \label{col4}\\
A+\frac12 e(f)&=0, \label{col5}\\
B+\frac12(\varphi e)(f)&=0. \label{col6}
\end{align}
In particular, subtracting \eqref{col1} from \eqref{col2} gives
\[
4a\ell=0.
\]
Since \(\ell>0\) on \(U\), we obtain
\[
a=0.
\]
\end{remark}

\begin{remark}[The case \(V\perp \xi\)]
\label{rem:V-orthogonal-xi}
If the potential vector field is orthogonal to the Reeb vector field, then
\[
V=f_1e+f_2\varphi e,
\]
so that \(f_3=0\). In this case, \eqref{S1}--\eqref{S6} reduce to
\begin{align}
\frac r2-1+\ell^2-2a\ell+e(f_1)-bf_2
&=\Lambda, \label{orth1}\\
\frac r2-1+\ell^2+2a\ell+(\varphi e)(f_2)-cf_1
&=\Lambda, \label{orth2}\\
2(1-\ell^2)
&=\Lambda, \label{orth3}\\
Z+\frac12\left\{e(f_2)+(\varphi e)(f_1)+bf_1+cf_2\right\}
&=0, \label{orth4}\\
A+\frac12\left\{\xi(f_1)+(1+\ell-a)f_2\right\}
&=0, \label{orth5}\\
B+\frac12\left\{\xi(f_2)+(a+\ell-1)f_1\right\}
&=0. \label{orth6}
\end{align}
\end{remark}

\begin{remark}[The case \(Q\xi=\sigma\xi\)]
\label{rem:Qxi-sigma-xi}
If
\[
Q\xi=\sigma\xi,
\]
then, by \eqref{pre-ric3},
\[
A=0,\qquad B=0,
\]
and
\[
\sigma=2(1-\ell^2).
\]
If, moreover, \(\sigma\) is constant along the Reeb vector field \(\xi\), then
\[
\xi(\sigma)=0.
\]
Since
\[
\sigma=2(1-\ell^2),
\]
we obtain
\[
\xi(\sigma)=-4\ell\,\xi(\ell)=-4\ell Z.
\]
On \(U\), \(\ell>0\), and hence
\[
Z=0.
\]
Therefore, under the assumptions
\[
Q\xi=\sigma\xi,\qquad \xi(\sigma)=0,
\]
the system \eqref{S1}--\eqref{S6} simplifies by setting
\[
A=B=Z=0.
\]
\end{remark}

\begin{theorem}
\label{thm:collinear-Qxi}
Let \((M^3,\varphi,\xi,\eta,g)\) be a three-dimensional contact metric
manifold satisfying
\[
Q\xi=\sigma\xi,
\]
where \(\sigma\) is a smooth function such that \(\xi(\sigma)=0\).
Assume that \(g\) defines an almost Ricci--Bourguignon soliton
\[
\operatorname{Ric}+\frac12\mathcal L_Vg=(\lambda+\rho r)g
\]
whose potential vector field is pointwise collinear with the Reeb vector field,
that is,
\[
V=f\xi
\]
for some smooth function \(f\). Then \(f=0\) on the non-Sasakian open set
\[
U=\{p\in M:h_p\neq0\}.
\]
Consequently, if \(V\) is nowhere vanishing on \(M\), then \(U=\varnothing\), and hence
\(M\) is \(K\)-contact. In particular, since \(\dim M=3\), \(M\) is Sasakian. Equivalently, under the condition \(\xi(\sigma)=0\), a non-Sasakian contact
metric three-manifold satisfying \(Q\xi=\sigma\xi\) admits no nowhere-vanishing
Reeb-aligned almost Ricci--Bourguignon potential.
\end{theorem}

\begin{proof}
On \(U\), choose a local orthonormal \(\varphi\)-basis
\[
\{e,\varphi e,\xi\}
\]
such that
\[
he=\ell e,\qquad h\varphi e=-\ell\varphi e,
\]
where \(\ell>0\).

Since
\[
Q\xi=\sigma\xi,
\]
the local Ricci operator formula gives
\[
A=0,\qquad B=0,
\]
and
\[
\sigma=2(1-\ell^2).
\]
Because \(\xi(\sigma)=0\), we have
\[
0=\xi(\sigma)=\xi\bigl(2(1-\ell^2)\bigr)
=-4\ell\,\xi(\ell).
\]
As \(\ell>0\) on \(U\), this implies
\[
\xi(\ell)=0.
\]
Thus
\[
Z=\xi(\ell)=0.
\]

Now let
\[
V=f\xi.
\]
In the component system of Proposition \ref{prop:component-system}, this
means
\[
f_1=f_2=0,\qquad f_3=f.
\]
The \((e,\varphi e)\)-component of the soliton equation becomes
\[
Z-\ell f=0.
\]
Since \(Z=0\) and \(\ell>0\) on \(U\), we get
\[
f=0
\]
on \(U\).

Since \(\xi\) is nowhere zero, \(V=f\xi\) is nowhere vanishing if and only if
\(f\neq0\) everywhere. As \(f=0\) on the open set \(U\), it follows that \(U=\emptyset\). Thus, the only possible Reeb-aligned almost Ricci--Bourguignon potential on the non-Sasakian region is the trivial one.
Hence \(h=0\) everywhere. Thus \(M\) is \(K\)-contact. Since every
three-dimensional \(K\)-contact manifold is Sasakian, the conclusion follows.
\end{proof}

\begin{corollary}
\label{cor:nonzero-collinear-sasakian}
Let \((M^3,\varphi,\xi,\eta,g)\) be a three-dimensional contact metric
manifold satisfying \(Q\xi=\sigma\xi\), where \(\xi(\sigma)=0\). If \(g\)
admits an almost Ricci--Bourguignon soliton whose potential vector field is
nowhere vanishing and pointwise collinear with \(\xi\), then \(M\) is Sasakian.
\end{corollary}

\begin{lemma}
\label{lem:constants-from-koufogiorgos}
Let \((M^3,\varphi,\xi,\eta,g)\) be a connected complete non-Sasakian contact
metric three-manifold satisfying
\[
Q\xi=\sigma\xi,
\]
where \(\sigma\) is constant. Then, on the adapted \(\varphi\)-basis
\(\{e,\varphi e,\xi\}\) with
\[
he=\ell e,\qquad h\varphi e=-\ell\varphi e,\qquad \ell>0,
\]
the functions \(\ell\), \(a\), and the scalar curvature \(r\) are constant.
\end{lemma}

\begin{proof}
Since \(M\) is non-Sasakian, the open set
\[
U=\{p\in M:h_p\neq0\}
\]
is nonempty. On \(U\), the condition \(Q\xi=\sigma\xi\), together with
\eqref{pre-ric3}, gives
\[
\sigma=2(1-\ell^2).
\]
Hence \(\ell\) is constant on \(U\). Since \(\ell>0\), we have
\[
\sigma<2.
\]
If \(h\) vanished at some point of \(M\), then at that point
\[
g(Q\xi,\xi)=2,
\]
and hence \(\sigma=2\), contradicting \(\sigma<2\). Therefore \(h\neq0\)
everywhere, and so \(U=M\).

Since \(Q\xi=\sigma\xi\) with \(\sigma\) constant is condition (i) in the main
theorem of Koufogiorgos \cite{koufogiorgos1995}, with
\(\sigma=\operatorname{Tr}l\), that theorem implies
that, in the non-Sasakian case, the Ricci operator has the form
\[
Q=\alpha I+\beta\,\eta\otimes\xi+\delta h,
\]
where \(\alpha,\beta,\delta\) are constants. Taking the trace gives
\[
r=3\alpha+\beta,
\]
because \(\operatorname{tr}h=0\) and \(\operatorname{tr}(\eta\otimes\xi)=1\).
Thus \(r\) is constant.

On the other hand, comparing the \(e\)- and \(\varphi e\)-components of \(Q\),
we have from \eqref{pre-ric1} and \eqref{pre-ric2}
\[
Qe=\left(\frac r2-1+\ell^2-2a\ell\right)e,
\]
and
\[
Q\varphi e=
\left(\frac r2-1+\ell^2+2a\ell\right)\varphi e.
\]
Meanwhile, from
\[
Q=\alpha I+\beta\,\eta\otimes\xi+\delta h
\]
we obtain
\[
Qe=(\alpha+\delta\ell)e,\qquad
Q\varphi e=(\alpha-\delta\ell)\varphi e.
\]
Subtracting the two component equations gives
\[
4a\ell=-2\delta\ell.
\]
Since \(\ell>0\), it follows that
\[
a=-\frac{\delta}{2}.
\]
Thus \(a\) is constant.
\end{proof}

\begin{remark}
The next theorem shows that, under the assumptions \(Q\xi=\sigma\xi\) and
\(\sigma=\mathrm{constant}\), the additional freedom arising from the variable
soliton function disappears in the Reeb-orthogonal non-Sasakian case. The proof
uses Koufogiorgos' classification theorem together with the reduced component
system derived above.
\end{remark}

\begin{theorem}
\label{thm:orthogonal-potential}
Let \((M^3,\varphi,\xi,\eta,g)\) be a connected complete non-Sasakian contact
metric three-manifold satisfying
\[
Q\xi=\sigma\xi,
\]
where \(\sigma\) is a constant. Suppose that \(g\) admits an almost
Ricci--Bourguignon soliton
\[
\operatorname{Ric}+\frac12\mathcal L_Vg=(\lambda+\rho r)g
\]
whose potential vector field is orthogonal to the Reeb vector field, that is,
\[
g(V,\xi)=0.
\]
Then
\[
\lambda+\rho r=\sigma.
\]
In particular, \(\lambda\) is constant, and the almost Ricci--Bourguignon soliton reduces to a Ricci--Bourguignon soliton. Moreover, \(M\) is Einstein.
If \(V\) is not identically zero, then \(M\) is flat.
\end{theorem}

\begin{proof}
Since \(M\) is non-Sasakian, the open set
\[
U=\{p\in M:h_p\neq0\}
\]
is nonempty. On \(U\), choose a local orthonormal \(\varphi\)-basis
\(\{e,\varphi e,\xi\}\) such that
\[
he=\ell e,\qquad h\varphi e=-\ell\varphi e,
\]
where \(\ell>0\).

Since \(Q\xi=\sigma\xi\), \eqref{pre-ric3} gives
\[
A=0,\qquad B=0,
\]
and
\[
\sigma=2(1-\ell^2).
\]
On \(U\), the identity \(\sigma=2(1-\ell^2)\), together with
\(\ell>0\), yields \(\sigma<2\). Since \(\sigma\) is constant on \(M\),
this inequality holds globally. If there existed \(p\in M\setminus U\),
then \(h_p=0\), and by \eqref{2.1b} we would have
\[
g(Q\xi,\xi)_p=2.
\]
On the other hand,
\[
g(Q\xi,\xi)=\sigma<2,
\]
which is a contradiction. Therefore \(h\neq0\) everywhere and hence
\[
U=M.
\]

Moreover, \(\ell\) is constant. Hence
\[
Z=\xi(\ell)=0.
\]
Using
\[
b=\frac{(\varphi e)(\ell)+A}{2\ell},\qquad
c=\frac{e(\ell)+B}{2\ell},
\]
we obtain
\[
b=c=0.
\]

By Lemma \ref{lem:constants-from-koufogiorgos}, the functions \(a\), \(\ell\), and the scalar curvature \(r\) are constant.

Because \(V\perp\xi\), locally we can write
\[
V=f_1e+f_2\varphi e.
\]
Thus \(f_3=0\). From the \((\xi,\xi)\)-component \eqref{S3} of the soliton
equation, we obtain
\[
\Lambda=\lambda+\rho r=2(1-\ell^2)=\sigma.
\]
Since \(\sigma\) and \(r\) are constant, it follows that
\[
\lambda=\sigma-\rho r
\]
is constant. Hence, the almost Ricci--Bourguignon soliton is actually a
Ricci--Bourguignon soliton.

Now the reduced component system becomes
\begin{align}
\frac r2-1+\ell^2-2a\ell+e(f_1)
&=\Lambda, \label{orth-proof1}\\
\frac r2-1+\ell^2+2a\ell+(\varphi e)(f_2)
&=\Lambda, \label{orth-proof2}\\
e(f_2)+(\varphi e)(f_1)
&=0, \label{orth-proof3}\\
\xi(f_1)+(1+\ell-a)f_2
&=0, \label{orth-proof4}\\
\xi(f_2)+(a+\ell-1)f_1
&=0. \label{orth-proof5}
\end{align}
Since \(a,\ell,r\) and \(\Lambda\) are constant, equations
\eqref{orth-proof1} and \eqref{orth-proof2} show that
\[
e(f_1)\quad\text{and}\quad (\varphi e)(f_2)
\]
are constant.

The Lie bracket relations reduce to
\begin{align}
[e,\varphi e]&=2\xi, \label{br1}\\
[e,\xi]&=-(a+\ell+1)\varphi e, \label{br2}\\
[\varphi e,\xi]&=(a-\ell+1)e. \label{br3}
\end{align}

Differentiating \eqref{orth-proof4} along \(e\), and using that \(a\) and
\(\ell\) are constant, gives
\[
e(\xi(f_1))+(1+\ell-a)e(f_2)=0.
\]
By \eqref{orth-proof3},
\[
e(f_2)=-(\varphi e)(f_1),
\]
and hence
\[
e(\xi(f_1))-(1+\ell-a)(\varphi e)(f_1)=0.
\]
On the other hand, applying \eqref{br2} to \(f_1\), and using that
\(e(f_1)\) is constant, gives
\[
e(\xi(f_1))-\xi(e(f_1))
=-(a+\ell+1)(\varphi e)(f_1).
\]
Since \(\xi(e(f_1))=0\), we obtain
\[
e(\xi(f_1))=-(a+\ell+1)(\varphi e)(f_1).
\]
Substituting this into the previous equation yields
\[
-2(\ell+1)(\varphi e)(f_1)=0.
\]
Since \(\ell>0\), it follows that
\[
(\varphi e)(f_1)=0.
\]
Therefore, by \eqref{orth-proof3},
\[
e(f_2)=0.
\]

Applying \eqref{br1} to \(f_1\), we obtain
\[
[e,\varphi e](f_1)=2\xi(f_1).
\]
But
\[
[e,\varphi e](f_1)
=e((\varphi e)(f_1))-(\varphi e)(e(f_1))=0,
\]
because \((\varphi e)(f_1)=0\) and \(e(f_1)\) is constant. Hence
\[
\xi(f_1)=0.
\]
Similarly, applying \eqref{br1} to \(f_2\), and using \(e(f_2)=0\) and
\((\varphi e)(f_2)\) constant, gives
\[
\xi(f_2)=0.
\]

Thus \eqref{orth-proof4} and \eqref{orth-proof5} reduce to
\begin{align}
(1+\ell-a)f_2&=0, \label{caseeq1}\\
(a+\ell-1)f_1&=0. \label{caseeq2}
\end{align}

We now distinguish the possible cases.

First suppose
\[
1+\ell-a=0.
\]
Then \(a=1+\ell\). Since \(\ell>0\),
\[
a+\ell-1=2\ell\neq0,
\]
and hence \eqref{caseeq2} gives
\[
f_1=0.
\]
Thus \(e(f_1)=0\). Applying \eqref{br2} to \(f_2\), and using
\(e(f_2)=\xi(f_2)=0\), we get
\[
0=[e,\xi](f_2)=-(a+\ell+1)(\varphi e)(f_2).
\]
Since \(a+\ell+1=2(\ell+1)\neq0\), this gives
\[
(\varphi e)(f_2)=0.
\]
Subtracting \eqref{orth-proof1} from \eqref{orth-proof2}, we obtain
\[
4a\ell+(\varphi e)(f_2)-e(f_1)=0.
\]
Therefore
\[
4a\ell=0.
\]
Since \(\ell>0\), we get \(a=0\), contradicting \(a=1+\ell\). Hence this
case cannot occur.

Next suppose
\[
a+\ell-1=0.
\]
Then \(a=1-\ell\). If \(\ell\neq1\), then
\[
1+\ell-a=2\ell\neq0,
\]
so \eqref{caseeq1} gives
\[
f_2=0.
\]
Thus \(e(f_2)=\xi(f_2)=0\), and from \eqref{orth-proof3} we obtain
\[
(\varphi e)(f_1)=0.
\]
Applying \eqref{br3} to \(f_1\), and using \(\xi(f_1)=0\), gives
\[
0=[\varphi e,\xi](f_1)=(a-\ell+1)e(f_1).
\]
Since \(a=1-\ell\), we have
\[
a-\ell+1=2(1-\ell).
\]
As \(\ell\neq1\), it follows that
\[
e(f_1)=0.
\]
Also \((\varphi e)(f_2)=0\), because \(f_2=0\). Subtracting
\eqref{orth-proof1} from \eqref{orth-proof2} gives again
\[
4a\ell=0.
\]
Since \(\ell>0\), we get \(a=0\). But \(a=1-\ell\), hence \(\ell=1\),
contradicting \(\ell\neq1\). Therefore \(\ell=1\), and consequently
\[
a=0.
\]

In this case
\[
\Lambda=2(1-\ell^2)=0.
\]
Moreover, from \eqref{caseeq1} we get
\[
(1+\ell-a)f_2=2f_2=0,
\]
so
\[
f_2=0.
\]
Equation \eqref{orth-proof2} then gives
\[
\frac r2=0,
\]
and hence
\[
r=0.
\]
Equation \eqref{orth-proof1} gives
\[
e(f_1)=0.
\]
Using \(\ell=1\), \(a=0\), \(r=0\), \(A=B=Z=0\), and \(b=c=0\), all Ricci
components vanish. Hence
\[
\operatorname{Ric}=0.
\]
Since \(M\) is three-dimensional, this implies that \(M\) is flat.

Finally, suppose neither
\[
1+\ell-a=0
\]
nor
\[
a+\ell-1=0
\]
holds. Then \eqref{caseeq1} and \eqref{caseeq2} imply
\[
f_1=f_2=0.
\]
Thus \(V=0\), and the soliton equation reduces to
\[
\operatorname{Ric}=\Lambda g.
\]
Hence \(M\) is Einstein.

Combining the above cases, \(M\) is Einstein. If \(V\) is not identically zero,
then only the flat case can occur.
\end{proof}

\begin{remark}
In the flat exceptional case appearing in the proof of
Theorem \ref{thm:orthogonal-potential}, one has \(\ell=1\), and hence
\[
\sigma=2(1-\ell^2)=0.
\]
This is the same value of \(\sigma=g(Q\xi,\xi)\) that appears in the study of
three-dimensional contact metric manifolds with vanishing Jacobi operator
\(l=R(\cdot,\xi)\xi\); see Koufogiorgos and Tsichlias
\cite{koufogiorgosTsichlias2009}.
\end{remark}

\begin{corollary}
\label{cor:orthogonal-flat}
Under the hypotheses of Theorem \ref{thm:orthogonal-potential}, if the
orthogonal potential vector field \(V\) is not identically zero, then \(M\) is flat.
Equivalently, every non-flat example under these hypotheses has \(V\equiv0\) and is Einstein.
\end{corollary}

\begin{proof}
This follows immediately from Theorem \ref{thm:orthogonal-potential}.
\end{proof}

\begin{example}
\label{ex:sphere}
Let \(S^3\) be the unit sphere with its standard Sasakian contact metric
structure \((\varphi,\xi,\eta,g)\). Then \(\xi\) is Killing,
\[
\operatorname{Ric}=2g,\qquad r=6.
\]
For any constant \(c\), take
\[
V=c\xi.
\]
Since \(\xi\) is Killing, \(\mathcal L_Vg=0\). Hence
\[
\operatorname{Ric}+\frac12\mathcal L_Vg
=
2g.
\]
Thus \((S^3,g,V,\lambda,\rho)\) is a Ricci--Bourguignon soliton provided
\[
\lambda=2-6\rho.
\]
This gives a trivial Reeb-aligned Ricci--Bourguignon soliton in the Sasakian
case.
\end{example}

\begin{remark}
The results show that the contact metric structure rigidly constrains the additional freedom introduced by allowing the soliton function to vary. Under the condition \(Q\xi=\sigma\xi\), the almost Ricci--Bourguignon equation either
forces a Reeb-aligned potential to vanish on the non-Sasakian region or forces
the soliton function to become constant in the Reeb-orthogonal non-Sasakian
case.
\end{remark}

\bibliographystyle{alpha}
	\bibliography{sn-bibliography}

\end{document}